\documentclass[11pt,a4paper]{amsart}
\usepackage{
amssymb,
amsmath}

\usepackage[bookmarks,colorlinks,pagebackref]{hyperref}

\numberwithin{equation}{section}
\newtheorem{theorem}{Theorem}[section]
\newtheorem{lemma}[theorem]{Lemma}
\newtheorem{proposition}[theorem]{Proposition}
\newtheorem{corollary}[theorem]{Corollary}

\theoremstyle{definition}

\newtheorem{problem}[theorem]{Problem}
\newtheorem{remark}[theorem]{Remark}

\newtheorem{notation}[theorem]{Notation}

\newcommand\Ext{\operatorname{Ext}}

\newcommand\Coker{\operatorname{Coker}}

\newcommand{\au}{\underline a}

\newcommand{\q}{\mathfrak{q}}

\author[Azeem Khadam]{M. Azeem Khadam}
\address{Abdus Salam School of Mathematical Sciences, GC University, Lahore, Pakistan}
\email{azeemkhadam@gmail.com}
\thanks{The author is grateful to DAAD and HEC, Pakistan for the support of his PhD research under 
	grant number 91524811 and 112-21480-2PS1-015 (50021731) respectively.}

\title[Multiplicity]{On Regular Sequences in the Form Module
	with Applications to Local B\'ezout Inequalities}

\begin{document}


\begin{abstract}
	Let $\q$ denote an ideal in a Noetherian local ring $(A,\mathfrak{m})$. Let $\au=a_1,\ldots,a_d \subset \q$ denote a system of parameters in a finitely generated $A$-module $M$. This note investigate an improvement of the inequality $c_1\cdot \ldots \cdot c_d \cdot e_0(\q;M) \leq \ell_A(M/\au\,M)$, where $c_i$ denote the initial degrees of $a_i$ in the form ring $G_A(\q)$. 
	To this end, there is an investigation of regular sequences in the form module $G_M(\q)$ by homology of a factor complex of the Koszul complex. In a particular case, there is a discussion of classical local B\'ezout inequality in the affine $d$-space $\mathbb{A}^d_k$.
\end{abstract}


\maketitle

\section{Introduction}

The importance of an improvement of the inequality $\ell_A(M / \au M) \geq c_1 \cdot \ldots \cdot c_d \cdot e_0(\q; M)$ has to do with B\'ezout's Theorem in the projective plane. 
Let $C = V(F), D = V(G) \subset \mathbb{P}^2_{\Bbbk}, \Bbbk = \overline{\Bbbk,}$ be two curves in the projective plane without a common component. Then 
\[
\sum_{P \in C\cap D} \mu(P;C,D) = \deg C \cdot \deg D,
\]
where $\mu(P;C,D)$ denotes the local intersection multiplicity of $P$ in $C \cap D$. In a particular case when $P$ is the origin, it follows that $\mu(P;C,D) = \ell_A(A / (f, g)A)$, where $A = \Bbbk[x,y]_{(x,y)}$ and $f,g$ denote 
the equations in $A$. Note that $\ell_A(A / (f,g)A) = e_0(f,g;A)$ as $A$ is a regular local ring. Since $C,D$ have no component in common, $\{f,g\}$ forms a system of parameters in $A$. Then 
\[
e_0(f,g;A) \geq c\cdot d \cdot e_0(\mathfrak{m};A) = c\cdot d,
\]
since $e_0(\mathfrak{m};A) = 1$, called the local B\'ezout inequality in the affine plane $\mathbb{A}_{\Bbbk}^2$. Here $c, d$ denote the initial degree of $f,g$ respectively. This 
estimate is well-known (see for instance \cite{BK} or \cite{F}) and proved by resultants or Puiseux expansions. 
Moreover, equality holds if and only if $C,D$ intersect transversally at the origin. In other words $f^{\star},g^{\star}$, 
the initial forms of $f,g$ in the form ring $G_A(\mathfrak{m}) \cong \Bbbk[X,Y]$, 
is a homogeneous system of parameters. 

First Byd\u{z}ovsk\'y \cite{B} and most recently Bo\u{d}a-Schenzel \cite{BSCH} presented an improvement of the local B\'ezout inequality. More precisely,
\[
e_0(f,g;A) \geq c \cdot d + t,
\]
where $t$ is the number of common tangents of $f,g$ at origin when counted with multiplicities.
We generalized their result to an arbitrary situation.
To this end, let $\q$ denote an ideal in a Noetherian local ring $(A, \mathfrak{m}, \Bbbk)$ such that $\ell_A(M / \q M)$ is finite for a finitely generated $A$-module $M$. Let $\au = a_1,\ldots,a_d \subseteq \q$ denote a system of parameters of $M$ such that $a_i \in \q^{c_i} \setminus \q^{c_i+1}$, $c_i > 0$, for $i=1,\ldots,d$. Then we have the following result.

\begin{theorem}\label{th-1}
		(Cor. \ref{cor-mul}) With the previous notations, if $\au^\star G_A(\q)$ contains a $G_M(\q)$-regular sequence $\underline{b}^\star=b_1^\star,\ldots,b_{d-1}^\star$
		and we choose $b_i$ for $i=1,\ldots,d-1$ as in Lemma \ref{lem4}. Then
		\[
		\ell_A(M/\au M) \geq c\cdot e_0(\q;M)+ \mathfrak{x}
		\]
		where $c=c_1\cdot \ldots \cdot c_d$ and $\mathfrak{x} = \ell_A([\Ext^{d - 1}_{G_A(\mathfrak{q})}(G_A(\mathfrak{q}) / \underline{a}^\star G_A(\mathfrak{q}) , G_M(\mathfrak{q}))]_{n-\overline{c}-1})$ is a constant for all $n \gg 0$ and $\overline{c}=c_1+\ldots+c_d.$
\end{theorem}

There are also few applications of the previous theorem. We refer Section 5 and 6.

Another motivation for the author was a recent preprint \cite{KSCH}. In this preprint, the authors define a generalized Koszul complex $\mathcal{L}_{\bullet}(\au,\q,M;n)$ which is factor complex of Koszul complex (see definition \ref{kos-1}). There are criteria concerning regular sequences in a finitely generated $A$-module $M$, which deal the vanishing and rigidity of the Koszul homology (see \cite{BH} and \cite{hM}). We present the similar criteria concerning regular sequences in the form module $G_M(\q)$ in terms of the homology modules $L_i(\au,\q,M;n)$ of the complex $\mathcal{L}_{\bullet}(\au,\q,M;n)$. More precisely, let $\au=a_1,\ldots,a_d$ and $\underline{b}=b_1,\ldots,b_t$ denote two systems of elements of $A$. There is a following theorem.

\begin{theorem}\label{th-2}
	(1)(Theorem \ref{th-reg}) With the previous notations, the following are equivalent:
	\begin{itemize}
		\item[(a)] $\au^\star=a_1^\star,\ldots,a_d^\star$ is $G_M(\q)$-regular sequence.
		\item[(b)] $L_1(\au,\q,M;n)=0$ for all $n$.
		\item[(c)] $L_i(\au,\q,M;n)=0$ for all $i>0$, for all $n$. 	
	\end{itemize}
	(2)(Theorem \ref{th-formula}, Prop. \ref{prop-formula})With the previous notations, if $\au^\star G_A(\q)$ contains a $G_M(\q)$-regular sequence $\underline{b}^\star=b_1^\star,\ldots,b_t^\star$, then
	\[
	L_i(\au,\q,M;n)=0 \text{ for all } i>t, \text{ for all } n.
	\]
	The converse is also true.
	Moreover, given $\underline{b}^\star=b_1^\star,\ldots,b_t^\star$ we choose $b_i$ for $i=1,\ldots,t$ as in Lemma \ref{lem4}, there is an isomorphism
	\[
	L_{d-t}(\au,\q,M;n) \cong \cap_{i=1}^d(\underline{b},\q^{n+\beta-\overline{c_i}})M :_M a_i / (\underline{b},\q^{n+\beta-\overline{c}})M.
	\]	
\end{theorem}

We refer Section 3 and 4 for the detail discussion about above theorem. In \cite{BSCH} and \cite{KSCH}, authors posed the problem to study the Euler characteristic $\chi_A(\au,\q,M)$ of the complex $K_{\bullet}(\au,\q,M;n)$, see def. \ref{kos-1}, independently of its value which is equal to $e_0(\au; M) - c_1\cdot \ldots \cdot c_d \cdot e_0(\q; M) \geq 0$ for $n \gg 0$, cf. \cite{BSCH}. In section 6, we discuss a few properties of this Euler characteristic. A further investigation of the geometric meanings of the length involved in Theorem \ref{th-1} in affine space $\mathbb{A}_{\Bbbk}^d$ when $d \geq 3$ is in progress.

As a source for basic notions in Commutative Algebra, we refer to \cite{AM} and \cite{hM}. For results 
on Homological Algebra, we refer to \cite{jR} and \cite{cW}.

\section{Preliminaries}

In this section, we present the basic notations, which we are going to use in upcoming sections. For more detail, we refer the text book \cite{hM} and lecture notes \cite{jpS}.

\begin{notation}\label{not1}
	(1) Let $(A, \mathfrak{m})$ be a local Noetherian ring, $\mathfrak{q}$ be an ideal in $A$ and $M$ be a finitely generated $A$-module. Then $\mathfrak{q}$ is said to be an ideal of definition with respect to $M$ if the length $\ell_A(M / \q M)$ of $A$-module $M / \mathfrak{q}M$ is finite. Now, it is easily seen that the length of $A$-modules $M / \q^nM$ is also finite for all $n \in \mathbb{N}$.
	
	For $n$ large enough, $\ell_A(M/\q^nM)$ becomes a polynomial, which is written as
	\[
	\ell_A(M/\q^nM) = \sum_{i=0}^d e_i(\q;M)\binom{n+d-i-1}{d-i},
	\]
	where degree $d$ is equal to $\dim M$ (see \cite{hM}).
	Here, $e_i(\q;M)$ are called the Hilbert-Samuel local multiplicities of $M$ with respect to $\q$. The first $e_0(\q;M)$ of them is our main ingredient for the rest of the note, and we call it just the multiplicity of $M$ with respect to $\q$.
	
	(2) The Rees and form rings of $A$ with respect 
	to $\mathfrak{q}$ are defined by 
	\[
	R_A(\mathfrak{q}) = \oplus_{n \geq 0} \mathfrak{q}^n \,T^n \subseteq A[T] \,
	\text{ and }\, G_A(\mathfrak{q}) = \oplus_{n \geq 0} \mathfrak{q}^n/\mathfrak{q}^{n+1},
	\]
	where $T$ denotes an indeterminate over $A$. 
	The Rees and form modules are defined in the corresponding way by 
	\[
	R_M(\mathfrak{q}) = \oplus_{n \geq 0} \mathfrak{q}^n M \,T^n \subseteq M[T] \,
	\text{ and } \, G_M(\mathfrak{q}) = \oplus_{n \geq 0} \mathfrak{q}^nM/\mathfrak{q}^{n+1}M.
	\]
	
	(3) Assume that $m \in M$ such that $m \in \mathfrak{q}^c M \setminus \mathfrak{q}^{c+1} M$. We define  
	$m^{\star} := m + \mathfrak{q}^{c+1} M \in [G_M(\mathfrak{q})]_c$. If $m \in \cap_{n \geq 1} \mathfrak{q}^n M$, 
	then we write $m^{\star} = 0$. Here $m^{\star}$ and $c$ are called the initial form and initial degree of $m$ in $G_M(\q)$ respectively. We refer \cite{SH} for more detail.
\end{notation}

For basic results about multiplicities, we refer \cite{BH} and \cite{hM}. Another tool for the investigation is use of Koszul complex.

\begin{remark} \label{kos} ({\sl Koszul Complex})
	Let $\underline{a} = a_1,\ldots,a_d$ denote a system of elements of the ring $A$. The Koszul complex 
	$K_{\bullet}(\underline{a};A)$ is defined as follows: Assume that $F$ is a free $A$-module with basis $e_1,\ldots,e_d$. 
	Then $K_i(\underline{a};A) = \bigwedge^i F$ for $i = 1,\ldots,d$. A basis of $K_i(\underline{a};A)$ is given by the wedge 
	products $e_{j_1} \wedge \ldots \wedge e_{j_i}$ for $1 \leq j_1 < \ldots < j_i \leq d$. The boundary 
	homomorphism $K_i(\underline{a};A) \to K_{i-1}(\au;A)$ is defined by 
	\[
	d_{j_1 \ldots j_i} :
	e_{j_1} \wedge \ldots \wedge e_{j_i} \mapsto \sum_{k=1}^{i} (-1)^{k+1} a_{j_k} e_{j_1}\wedge \ldots \wedge \widehat{e_{j_k}} 
	\wedge \ldots \wedge e_{j_i}
	\] 
	on the free generators $e_{j_1} \wedge \ldots \wedge e_{j_i}$. Also $K_{\bullet}(\au;M) \cong K_{\bullet}(\au;A) \otimes_A M$. We write $H_i(\au;M), i \in \mathbb{Z},$ for the $i$-th homology of $K_{\bullet}(\au;M)$.
	
\end{remark}

For more detail about Koszul homology, we refer \cite{BH} and \cite{hM}. The following are main ingredients for the investigation, see \cite{KSCH} for reference. 

\begin{notation}\label{kos-1}(Khadam-Schenzel \cite{KSCH})
	(1) Assume that $a_i \in \q^{c_i}$ for $i=1,\ldots,d$ and $n$ is a non-negative integer. For $n < 0$, we assume that $\mathfrak{q}^n M = 0$. We define a complex $K_{\bullet}(\au,\mathfrak{q},M;n)$ in the following way:
	\begin{itemize}
		\item[(i)] The $i$-th term $K_i(\au,\mathfrak{q},M;n) := \oplus_{1\leq j_1 < \ldots < j_i \leq d} 
		\mathfrak{q}^{n-c_{j_1}- \ldots- c_{j_i}}M$ for $0 \le i \le d$ and $K_i(\au,\mathfrak{q},M;n) = 0$ otherwise.
		\item[(ii)] The boundary homomorphism $K_i(\au,\mathfrak{q},M;n) \to K_{i-1}(\au,\mathfrak{q},M;n)$ is defined by homomorphisms on each of 
		the direct summands $\mathfrak{q}^{n-c_{j_1}- \ldots- c_{j_i}}M$. On $\mathfrak{q}^{n-c_{j_1}- \ldots- c_{j_i}}M$, 
		it is the map given by $d_{j_1 \ldots j_i} \otimes 1_M$ restricted to $\mathfrak{q}^{n-c_{j_1}- \ldots- c_{j_i}}M$, where $d_{j_1 \ldots j_i}$ denotes the homomorphism as defined in \ref{kos}.
	\end{itemize}
	It is clear that $K_{\bullet}(\au,\q,M;n)$ is a complex. 
	Moreover, by construction, $K_{\bullet}(\au,\mathfrak{q},M;n)$ is a sub complex of the Koszul complex $K_{\bullet}(\au;M)$. We write $H_i(\au,\mathfrak{q},M;n)$,$i \in \mathbb{Z}$, for the $i$-th 
	homology of the complex $K_{\bullet}(\au,\mathfrak{q},M;n)$.
	Note that $[K_{\bullet}(\underline{aT^c};R_M(\mathfrak{q}))]_n 
	\cong K_{\bullet}(\au,\mathfrak{q},M;n)$ for $n \in \mathbb{N}$.
	
	(2) We define $\mathcal{L}_{\bullet}(\au,\mathfrak{q},M;n)$ as the quotient of the 
	embedding $K_{\bullet}(\au,\mathfrak{q},M;n) \to K_{\bullet}(\au;M)$. That is, there is a short exact sequence of complexes
	\[
	0 \to K_{\bullet}(\au,\mathfrak{q},M;n) \to K_{\bullet}(\au;M) \to \mathcal{L}_{\bullet}(\au,\mathfrak{q},M;n) 
	\to 0.
	\]
	Note that $\mathcal{L}_i(\au,\mathfrak{q},M;n) \cong \oplus_{1 \leq j_1 < \ldots < j_i \leq d}M/\mathfrak{q}^{n-c_{j_1}- \ldots- c_{j_i}}M$.
	The boundary homomorphisms are those induced by the Koszul complex. We write $L_i(\au,\mathfrak{q},M;n)$, $i \in \mathbb{Z}$, 
	for the $i$-th homology of the complex $\mathcal{L}_{\bullet}(\au,\mathfrak{q},M;n)$.
\end{notation}

For more detail about complexes $K_{\bullet}(\au,\q,M;n)$ and $\mathcal{L}_{\bullet}(\au,\q,M;n)$, and their relationship with the local cohomology module, we refer Khadam-Schenzel \cite{KSCH}.

\section{Regular Sequences in the Form and Rees Modules}

There is a criterion in terms of Koszul homology which ensures whether a sequence of elements in $A$ is $M$-regular or not. More precisely, the sequence $\au=a_1,\ldots,a_d$ is $M$-regular if and only if $H_i(\au;M)=0$ for all $i > 0$ 
 (see \cite{hM}). In this section, we present a similar criterion in the form module $G_M(\q)$ in terms of the homology modules $L_i(\au,\q,M;n)$.

Let $\au^\star=a_1^\star,\ldots,a_d^\star$ denote a sequence of initial forms in the form ring $G_A(\q)$ with $\deg a_i^\star = c_i$ for $i=1,\ldots,d$. We start with the main result of the section.

\begin{theorem}\label{th-reg}
	With the previous notations, the following are equivalent:
\begin{itemize}
	\item[(1)] $\au^\star=a_1^\star,\ldots,a_d^\star$ is $G_M(\q)$-regular sequence.
	\item[(2)] $L_1(\au,\q,M;n)=0$ for all $n$.
	\item[(3)] $L_i(\au,\q,M;n)=0$ for all $i>0$, for all $n$. 	
\end{itemize}	
\end{theorem}

\begin{proof}
	For (1) $\Rightarrow$ (3), we use induction on $d$. If $d=1$, then 
	\[
	L_1(a_1,\q,M;n) = \q^nM :_A a_1 / \q^{n-c_1}M,
	\]
	which is equal to zero for all $n$ if and only if $a_1^\star$ is $G_M(\q)$-regular. Now, by virtue of long exact homology sequence coming from the mapping cone construction of the complex $\mathcal{L}_{\bullet}(\au,\q,M;n)$ (see \cite{KSCH}) and by inductive step, $L_i(\au,\q,M;n) = 0$ for all $i>1$, for all $n$, and 
	\[
	L_1(\au,\q,M;n) \cong 0 :_{L_0(\au',\q,M;n-c_d)} a_d.
	\]
	The latter is isomorphic to $(\au',\q^n)M :_A a_d / (\au',\q^{n-c_d})M$, which is equal to zero for all $n$, see \cite{VV}. It is obvious that (3) $\Rightarrow$ (2).
	
	For (2) $\Rightarrow$ (1), we apply induction on $d$ once again. The case $d=1$ is clear, see above. Again, from the mapping cone construction and by assumption, 
	\[
	L_1(\au',\q,M;n) = a_dL_1(\au',\q,M;n-c_d),
	\]
	and
	therefore, by virtue of Nakayama lemma $L_1(\au',\q,M;n) = 0 $ for all $n$. Hence $\au'^\star=a_1^\star,\ldots,a_{d-1}^\star$ is a $G_M(\q)$-regular sequence by induction. Moreover, 
	\[
	0=L_1(\au,\q,M;n) \cong (\au',\q^n)M :_A a_d / (\au',\q^{n-c_d})M
	\]
	for all $n$, and hence $a_d^\star$ is $G_M(\q)/(\au'^\star)G_M(\q)$-regular. This finishes the argument.
\end{proof}

The following is a consequence of the previous theorem.

\begin{corollary}\label{cor2}
	With the previous notations, $\au^\star=a_1^\star,\ldots,a_d^\star$ is a $G_M(\q)$-regular sequence implies that $\underline{aT^c}=a_1T^{c_1},\ldots,a_dT^{c_d}$ is an $R_M(\q)$-regular sequence.
\end{corollary}

\begin{proof}
	If $\au^\star=a_1^\star,\ldots,a_d^\star$ is a $G_M(\q)$-regular sequence, then $\au=a_1,\ldots,a_d$ is an $M$-regular sequence, see \cite{VV}. Hence $H_i(\au;M) = 0$ for all $i>0$. Therefore, from the long exact sequence of homology coming from the short exact sequence of \ref{kos-1}, $H_i(\au,\q,M;n)=0$ for all $i>0$ and for all $n$ (see \ref{th-reg}). That is $H_i(\underline{aT^c}; R_M(\q))=0$ for all $i>0$. Hence, by virtue of Koszul criterion, $\underline{aT^c}=a_1T^{c_1},\ldots,a_dT^{c_d}$ is an $R_M(\q)$-regular sequence.
\end{proof}

\section{A Formula for Homology}

There is a classical result concerning the length of an $M$-sequence inside the ideal $(a_1,\ldots,a_d)$ and vanishing of the Koszul homology. More precisely, if $b_1,\ldots,b_t$ is an $M$-sequence contained in the ideal $(a_1,\ldots,a_d)$, then $H_i(\au;M)=0$ for all $i>d-t$, and there is a formula 
\[
H_{d-t}(\au;M) \cong (b_1,\ldots,b_t)M :_A \au / (b_1,\ldots,b_t)M,
\]
see \cite{BH}. In this section, we present the similar result for the homology modules $L_i(\au,\q,M;n)$.

We begin with a lemma.

\begin{lemma}\label{lem3}
	With the previous notations, let $b^\star$ be a $G_M(\q)$-regular element of degree $\beta$, then there is a short exact sequence of complexes
	\[
	0 \to \mathcal{L}_{\bullet}(\au,\q,M;n-\beta) \stackrel{b} \to \mathcal{L}_{\bullet}(\au,\q,M;n) \to \mathcal{L}_{\bullet}(\au,\q,M/bM;n) \to 0.
	\]
	In particular, there is the long exact homology sequence
	\[
	\ldots \to L_i(\au,\q,M;n-\beta) \stackrel{b} \to L_i(\au,\q,M;n) \to L_i(\au,\q,M/bM;n) \to \ldots.
	\]
\end{lemma}

\begin{proof}
	The kernel of the map $\mathcal{L}_{\bullet}(\au,\q,M;n-\beta) \stackrel{b} \to \mathcal{L}_{\bullet}(\au,\q,M;n)$ is zero since $b^\star$ is $G_M(\q)$-regular. Also, it is easy to see that $\Coker (m_b) = \mathcal{L}_{\bullet}(\au,\q,M/bM;n)$. This provides the short exact sequence of complexes. By taking homology 
	it yields the long exact sequence.
\end{proof}	

Let $\underline{b}=b_1,\ldots,b_t$ denote a sequence of elements in $A$, and $\underline{b}^\star=b_1^\star,\ldots,b_t^\star$ denote a sequence of initial forms in $G_A(\q)$ with $\deg b_i^\star = \beta_i$. There is another technical lemma.

\begin{lemma}\label{lem4}
    With the previous notations, assume that
	$\underline{b}^\star G_A(\mathfrak{q}) \subseteq \au^\star G_A(\mathfrak{q})$.
    Then there are elements $b_1', \cdots, b_t' \in A$ such that
\begin{itemize}
	\item[(i)] $b_i^\star = b_i'^\star$ for $i = 1, \cdots, t$.
	\item[(ii)] $(b_1', \cdots, b_t')A \subseteq (a_1, \cdots, a_d)A$.
\end{itemize}	
\end{lemma}

\begin{proof}
	The containment relation of the assumption restricted to degree $n \in \mathbb{Z}$ provides
	$$
	(\sum_{i = 1}^tb_i\mathfrak{q}^{n - \beta_i} + \mathfrak{q}^{n + 1}) / \mathfrak{q}^{n + 1} \subseteq (\sum_{j = 1}^da_j\mathfrak{q}^{n - c_j} + \mathfrak{q}^{n + 1}) / \mathfrak{q}^{n + 1}
	$$ 
	for all $n$, and hence
	$\sum_{i = 1}^tb_i\mathfrak{q}^{n - \beta_i} \subseteq \sum_{j = 1}^da_j\mathfrak{q}^{n - c_j} + \mathfrak{q}^{n + 1}$ for all $n$.
	Now choose $n = \beta_k$ and therefore 
    $b_k \in \sum_{j = 1}^da_j\mathfrak{q}^{\beta_k - c_j} + \mathfrak{q}^{\beta_k + 1}$.
	Whence there exist $r_{jk} \in \mathfrak{q}^{\beta_k - c_j}$ for $j = 1, \cdots, d,$
	such that $b_k - \sum_{j = 1}^da_jr_{jk} \in \mathfrak{q}^{\beta_k + 1}$.
	Note that $\sum_{j = 1}^da_jr_{jk} \in \mathfrak{q}^{\beta_k} \setminus \mathfrak{q}^{\beta_k + 1}$.
	We choose $b_k' = \sum_{j = 1}^da_jr_{jk}$ for $k = 1, \cdots, t$, and this finishes the proof.
\end{proof}

Now, we present the main result of the section.

\begin{theorem}\label{th-formula}
	With the previous notations, if $\au^\star G_A(\q)$ contains a $G_M(\q)$-regular sequence $\underline{b}^\star=b_1^\star,\ldots,b_t^\star$, then
	\[
	L_i(\au,\q,M;n)=0 \text{ for all } i>d-t, \text{ for all } n.
	\]
	Moreover, given $\underline{b}^\star=b_1^\star,\ldots,b_t^\star$ we choose $b_i$ for $i=1,\ldots,t$ as in Lemma \ref{lem4}, there is an isomorphism
	\[
	L_{d-t}(\au,\q,M;n) \cong \cap_{i=1}^d(\underline{b},\q^{n+\beta-\overline{c_i}})M :_M a_i / (\underline{b},\q^{n+\beta-\overline{c}})M,
	\]
	where $c:=c_1\cdot \ldots \cdot c_d, \overline{c}:=\sum_{i=1}^dc_i, \overline{c_i}:=c_1+\ldots+c_{i-1}+c_{i+1}+\ldots+c_d,$ and $\beta:=\sum_{j=1}^t\beta_j$.
\end{theorem}

\begin{proof}
     We proceed by induction on $t$. The vanishing $L_i(\au,\q,M;n) = 0$ for $i>d$ is trivial, and it is easily seen that $L_d(\au,\q,M;n) \cong \cap_{i=1}^d \q^{n-\overline{c_i}}M :_Ma_i / \q^{n-\overline{c}}M$.
     Now assume that $t>0$ and $\underline{b}A \subseteq \au A$ by lemma \ref{lem4}. As by virtue of Valla-Valabrega \cite{VV}, $b_2^\star,\ldots,b_t^\star$ is a $G_{M/b_1M}(\q)$-regular sequence, therefore by induction $L_i(\au,\q,M/b_1M;n)=0$ for all $i>d-t+1$ and for all $n$. Hence by lemma \ref{lem3}, $L_i(\au,\q,M;n)=0$ for all $i>d-t+1$ and for all $n$, and $L_{d-t+1}(\au,\q,M;n-\beta_1)=0$ for all $n$. Note that $\underline{b}L_i(\au,\q,M;n)=0$ for all $i$, for all $n$, see \cite[Theorem 3.5(b)]{KSCH}.
     
     Again by induction 
     \[
     L_{d-t+1}(\au,\q,M/b_1M;n) \cong \cap_{i=1}^d(\underline{b},\q^{n+\sum_{j=2}^t\beta_j-\overline{c_i}})M :_Ma_i/(\underline{b},\q^{n+\sum_{j=2}^t\beta_j-\overline{c}})M
     \]
     and hence by using \ref{lem3} and $b_1L_i(\au,\q,M;n) = 0$ for all $i,n$, we get 
     \[
     L_{d-t+1}(\au,\q,M/b_1M;n)\cong L_{d-t}(\au,\q,M;n-\beta_1).
     \]
     This finishes the inductive argument.
\end{proof}

There is a converse of the previous theorem.

\begin{proposition}\label{prop-formula}
	With the previous notations, assume that $L_i(\au,\q,M;n)=0$ for all $i>d-t$, for all $n$, then $\au^\star G_A(\q)$ contains a $G_M(\q)$-regular sequence $\underline{b}^\star=b_1^\star,\ldots,b_t^\star$.
\end{proposition}

\begin{proof}
	Note that from the short exact sequence
	\[
	0 \to \q^nM/\q^{n+1}M \to M / \q^{n+1}M \to M / \q^nM \to 0,
	\]
	there is the following short exact sequence of complexes
	\[
	0 \to K_{\bullet}(\au^\star;G_M(\q))_n \to \mathcal{L}_{\bullet}(\au,\q,M;n+1) \to \mathcal{L}_{\bullet}(\au,\q,M;n) \to 0
	\]
	for $n\in \mathbb{N}$, where $K_{\bullet}(\au^\star;G_M(\q))_n$ denotes the $n$th component of the Koszul complex of $G_M(\q)$ w.r.t $\au^\star = a_1^\star,\ldots,a_d^\star$. From here, by view of long homology exact sequence $H_i(\au^\star;G_M(\q))_n=0$ for all $i>d-t$, for all $n$, and hence $H_i(\au^\star;G_M(\q))=0$ for all $i>d-t$. Now the result follows by virtue of Koszul homology, see \cite{hM}.
\end{proof}	

\section{Applications}

Let $(A,\mathfrak{m})$ denote a local Noetherian ring and $M$ be a finitely generated $A$-module with $\dim M = d$. Let $\au=a_1,\ldots,a_d$ denote a system of parameters of $M$ such that $\au \subset \q$.
We present the main result of the section.

\begin{theorem}\label{th-mul}
	With the previous notations, if $\au^\star G_A(\q)$ contains a $G_M(\q)$-regular sequence of length $d-1$, then
	\begin{itemize}
		\item[(1)] $\ell_A(L_1(\au,\q,M;n))$ is a constant for all $n \gg 0$.
		\item[(2)] $\ell_A(M/\au\,M)= c \cdot e_0(\q;M)+ \ell_A(L_1(\au,\q,M;n))$ for all $n \gg 0$, where $c=c_1 \cdot\ldots \cdot c_d$.
	\end{itemize}
\end{theorem}

\begin{proof}
	Note that the alternating sum of the lengths of modules in the complex $\mathcal{L}_{\bullet}(\au, \mathfrak{q},M; n)$ is
	\[
	\sum_{i=0}^{d} (-1)^i \sum_{1 \leq j_1 < \ldots < j_i \leq d}
	\ell_A(M/\mathfrak{q}^{n-c_{j_1}-\ldots- c_{j_i}}M),
	\]
	which is a weighted $d$-fold difference operator of Hilbert-Samuel polynomial
	for all $n \gg 0$ and hence is a constant $c \cdot e_0(\q;M)$. Also, it coincides with the Euler characteristic
	\[
	\chi_A (\mathcal{L}_{\bullet}(\au, \mathfrak{q},M; n)) = \sum_{i \geq 0}(-1)^i \ell_A(L_i(\au,\q,M;n)) \text{ for all } n \gg 0,
	\]
	see \cite{KSCH} for more detail.
    As $L_0(\au, \mathfrak{q},M; n) \cong M / (\au M, \q^nM) = M / \au M$ for all $n \gg 0$ since $\q^nM \subseteq \au M$, therefore (2) follows from theorem \ref{th-formula}. Also, (1) follows from (2). This completes the argument.
\end{proof}

Now, we describe the length $\ell_A(L_1(\au,\q,M;n))$.

\begin{proposition}\label{prop-mul}
	With the previous notations, if $\au^\star G_A(\q)$ contains a $G_M(\q)$-regular sequence $\underline{b}^\star=b_1^\star,\ldots,b_{d-1}^\star$
	 such that $\deg b_i^\star=\beta_i$ and we choose $b_i$ for $i=1,\ldots,d-1$ as in Lemma \ref{lem4}. Then
	$\ell_A (L_1(\au, \mathfrak{q},M; n))$ might be broken into two pieces.
	That is,
	\[
	\ell_A(L_1(\au, \mathfrak{q},M; n)) = \ell_A([\underline{b}^\star G_M(\mathfrak{q})
	: \underline{a}^\star / (\underline{b}^\star G_M(\mathfrak{q}))]_{n + \beta - \overline{c}-1}) + \ell_n,
	\]
	where 
	$$
	\ell_n = \ell_A (\cap_{i = 1}^d(\underline{b}, \mathfrak{q}^{n + \beta - \overline{c}_i})M :_M a_i / (\cap_{i = 1}^d(\underline{b}, \mathfrak{q}^{n + \beta - \overline{c}_i})M :_M a_i) \cap (\underline{b}, \mathfrak{q}^{n + \beta - \overline{c} - 1})M),
	$$ 
	with
	$ c=c_1 \ldots c_d, \overline{c}=\sum_{i=1}^dc_i, \overline{c_i}=c_1+\ldots+c_{i-1}+c_{i+1}+\ldots+c_d$,
	and $\beta=\sum_{j=1}^{d-1}\beta_j$.
	
	Moreover, for $n \gg 0$, all of the lengths involved here are constants and independent
	of the choice of $\underline{b}^\star$. We write $\mathfrak{x} = \ell_A([\underline{b}^\star G_M(\mathfrak{q})
	: \underline{a}^\star / (\underline{b}^\star G_M(\mathfrak{q}))]_n)$ and $\ell = \ell_n$ for $n \gg 0$.
\end{proposition}

\begin{proof}
	As $\underline{b}^\star$ is a $G_M(\q)$-regular sequence, hence $G_M(\mathfrak{q}) / (\underline{b}^\star)G_M(\q) \cong G_{M / \underline{b}M}(\mathfrak{q})$,
	see \cite{VV}. Therefore,
	it is easily seen that
	\[
	[\underline{b}^\star G_M(\mathfrak{q}) : \underline{a}^\star / \underline{b}^\star G_M(\mathfrak{q})]_n \cong (\cap_{i = 1}^d(\underline{b}, \mathfrak{q}^{n + c_i + 1})M :_Ma_i) \cap (\underline{b}, \mathfrak{q}^n)M / (\underline{b}, \mathfrak{q}^{n + 1})M.
	\]
	Now, we have the following short exact sequence
	\begin{gather*}
		0 \to [\underline{b}^\star G_M(\mathfrak{q}) : \underline{a}^\star / \underline{b}^\star G_M(\mathfrak{q})]_{n + \beta - \overline{c} - 1}
		\to L_1(\au, \mathfrak{q},M; n) \to \\ \to (\cap_{i = 1}^d(\underline{b}, \mathfrak{q}^{n + \beta - \overline{c}_i})M :_M a_i) / (\cap_{i = 1}^d(\underline{b}, \mathfrak{q}^{n + \beta - \overline{c}_i})M :_M a_i) \cap (\underline{b}, \mathfrak{q}^{n + \beta - \overline{c} - 1})M \to 0,
	\end{gather*}
	see theorem \ref{th-formula}.
	By counting the lengths, it provides the first
	equality of the statement. The length of the
	module in the middle
	is constant for $n \gg 0$, see \ref{th-mul}. Also, the length of the module in the left is constant for $n \gg 0$ since it is of dimension 1. By comparing the Hilbert polynomials, this
	proves that all the lengths are constants for all $n \gg 0$. 
	
	Note that
	\[
	\underline{b}^\star G_M(\mathfrak{q}) : \underline{a}^\star / \underline{b}^\star G_M(\mathfrak{q}) \cong \Ext^{d - 1}_{G_A(\mathfrak{q})}(G_A(\mathfrak{q}) / \underline{a}^\star G_A(\mathfrak{q}) , G_M(\mathfrak{q}))[-\beta].
	\]
	Therefore, we conclude that
	$\ell_A([\underline{b}^\star G_M(\mathfrak{q})
	: \underline{a}^\star / \underline{b}^\star G_M(\mathfrak{q})]_n)$
	is independent of the choice of $\underline{b}^\star$, and consequently,
	$\ell$ is also independent of the choice of $\underline{b}^\star$.
	This completes the proof.
\end{proof}

Now, we have the main result of the section, which is also the consequence of previous two results.

\begin{corollary}\label{cor-mul}
	With the previous notations, if $\au^\star G_A(\q)$ contains a $G_M(\q)$-regular sequence $\underline{b}^\star=b_1^\star,\ldots,b_{d-1}^\star$
	and we choose $b_i$ for $i=1,\ldots,d-1$ as in Lemma \ref{lem4}. Then
	\[
	\ell_A(M/\au M) \geq c\cdot e_0(\q;M)+ \mathfrak{x}
	\]
	where $c=c_1 \cdot \ldots \cdot c_d$ and $\mathfrak{x} = \ell_A([\Ext^{d - 1}_{G_A(\mathfrak{q})}(G_A(\mathfrak{q}) / \underline{a}^\star G_A(\mathfrak{q}) , G_M(\mathfrak{q}))]_{n-\overline{c}-1})$ is a constant for all $n \gg 0$ and $\overline{c}=c_1+\ldots+c_d.$
\end{corollary}

We mention a geometric application to local B\'ezout inequality in the affine plane $\mathbb{A}_k^2$.

\begin{remark}\label{Bezout}
	Let $\Bbbk$ be an algebraically close field and $A = \Bbbk[x,y]_{(x,y)}$ be a local ring. Also, let $f,g$ denote a system of parameters in $A$
	and $\mathfrak{m}$ denote the maximal ideal of $A$. Then $B := \Bbbk[X,Y] \cong G_A(\mathfrak{m})$ and $1 = e_0(\mathfrak{m};A)$. Then the above two results imply that
	\[
	e_0(f,g;A) \geq c \cdot d + t,
	\]
	where $t$ denotes the number of common tangents to $f,g$ at origin when counted with multiplicities. Note that $\ell_A([f^\star B :_B g^\star / f^\star B]_n) = t$ for all $n \gg 0$ (see \cite{BSCH}).
\end{remark}

\begin{problem}\label{Bezout-problem}
	Let $M=A=k[x_1,\ldots,x_d]_{(x_1,\ldots,x_d)}$ be the local ring and $\q=\mathfrak{m}=(x_1,\ldots,x_d)A$, where $d \geq 3$. Let $\au'^\star=a_1^\star,\ldots,a_{d-1}^\star$, then the author does not know the geometric interpretation of
	\[
	\mathfrak{x}= \ell_A([\Ext^{d - 1}_{G_A(\mathfrak{q})}(G_A(\mathfrak{q}) / \underline{a}^\star G_A(\mathfrak{q}) , G_M(\mathfrak{q}))]_n) = \ell_A([\au'^\star G_A(\mathfrak{m}) : a_d^\star / \au'^\star G_A(\mathfrak{m})]_{n+c_1+\ldots+c_{d-1}})
	\]
	for all $n \gg 0$. This problem can be related to the homological terms as in case of $d=2$.
\end{problem}

In the next, we present another consequence. More precisely, there is an upper bound to $\ell_A(M/\au\,M) - e_0(\au;M)\geq 0.$

\begin{corollary}\label{cor-ub1}
	With the previous notations, if $\au^\star G_A(\q)$ contains a $G_M(\q)$-regular sequence $\underline{b}^\star=b_1^\star,\ldots,b_{d-1}^\star$, then
	\begin{itemize}
		\item[(1)] $\ell_A(M/\au\,M) - e_0(\au;M) \leq \ell_A(L_1(\au,\q,M;n))$ for all $n \gg 0$.
		\item[(2)] Equality occurs when $\au^\star$ is a system of parameters of $G_M(\q)$. The converse is not true in general.
	\end{itemize}
\end{corollary}

\begin{proof}
	Since $c \cdot e_0(\mathfrak{q}; M) \le e(\underline{a}; M)$ (see \cite{BSCH}).
	Therefore claim in (1) follows from previous theorem \ref{th-mul}.
	Note that
	\[
	\ell_A(M / \underline{a}M) - e_0(\underline{a}; M) = \ell_A(L_1(\underline{a}, \mathfrak{q},M; n)) \text{ for all } n \gg 0
	\Leftrightarrow e_0(\underline{a}; M) = c \cdot e_0(\mathfrak{q}; M),
	\]
	(see theorem \ref{th-mul}).
	Now, the claim in (2) follows
	by \cite[Theorem 5.1]{BSCH}.
\end{proof}

\section{An Euler Characteristic}

With the notations of the previous section, we have the following lemma.

\begin{lemma}\label{cm1}
	With the previous notations, if $\au'^\star=a_1^\star,\ldots,a_{d-1}^\star$ is a $G_M(\q)$-regular sequence, then
	\begin{gather*}
	e_0(\au; M) = c\cdot e_0(\q;M) + \ell_A(\q^nM / \sum \limits_{i = 1}^d a_i\q^{n - c_i}M) \\ - \ell_A(\sum \limits_{i = 1}^{d-1} a_i\q^{n + c_d - c_i}M :_M a_d \cap \q^nM / \sum \limits_{i = 1}^{d-1} a_i\q^{n - c_i}M)
	\end{gather*}
	for all $n \gg 0$, where $c=c_1\cdot \ldots \cdot c_d$.
\end{lemma}

\begin{proof}
	Note $H_i(\au;M)=0$ for all $i > 1$ since $a_1,\ldots,a_{d-1}$ is $M$-regular sequence, cf. \cite{VV}. Also, $L_i(\au,\q,M;n)=0$ for all $i>1$ and for all $n$, see \ref{th-formula}. Moreover,
	\[
	L_0(\au,\q,M;n) \cong M / (\au,\q^n)M \cong M/\au M = H_0(\au;M)
	\]
	since $\q^nM \subseteq \au M$
	for $n \gg 0$. Therefore, from the long exact homology sequence coming from the short exact sequence in \ref{kos-1}, we get the following exact sequence
	\[
	0 \to H_1(\au,\q,M;n) \to H_1(\au;M) \to L_1(\au,\q,M;n) \to H_0(\au,\q,M;n) \to 0,
	\]
	for all $n \gg 0$. Note that $H_1(\au;M) \cong \au'M :_M a_d/\au'M$, where $\au'=a_1,\ldots,a_{d-1}$. By Theorem \ref{th-mul}, we get
	\[
	\ell_A(M / \au M)-\ell_A(\au'M :_M a_d/\au'M) = c\cdot e_0(\q;M) + \ell_A(H_0(\au,\q,M;n)) - \ell_A(H_1(\au,\q,M;n))
	\]
	for all $n \gg 0$. Finally, note that 
	\[
	\ell_A(M / \au M)-\ell_A(\au'M :_M a_d/\au'M) = e_0(\au;M), H_0(\au,\q,M;n) = \q^nM / \sum \limits_{i = 1}^d a_i\q^{n - c_i}M
	\]
	and
	\[
	H_1(\au,\q,M;n) = [H_1(\underline{aT^c};R_M(\q))]_n \cong \sum \limits_{i = 1}^{d-1} a_i\q^{n + c_d - c_i}M :_M a_d \cap \q^nM / \sum \limits_{i = 1}^{d-1} a_i\q^{n - c_i}M
	\]
	since $a_1T^{c_1},\ldots,a_{d-1}T^{c_{d-1}}$ is an $R_M(\q)$-regular sequence, see \ref{cor2}. This finishes the proof.
\end{proof}

Let $\chi_A(\au,\q,M;n)$ denote the Euler characteristic of the complex $K_{\bullet}(\au,\q,M;n)$. With the assumption of previous lemma,
\[
\chi_A(\au,\q,M;n) = \ell_A(\q^nM / \sum \limits_{i = 1}^d a_i\q^{n - c_i}M) - \ell_A(\sum \limits_{i = 1}^{d-1} a_i\q^{n + c_d - c_i}M :_M a_d \cap \q^nM / \sum \limits_{i = 1}^{d-1} a_i\q^{n - c_i}M)
\]
which is a constant. Even in a more general situation, we have
\[
\chi_A(\au,\q,M;n) = e_0(\au;M)-c\cdot e_0(\q;M) \text{ for all } n \gg 0,
\]
see \cite{BSCH} or \cite{KSCH}. Moreover, the authors mentioned a problem of giving an interpretation to $\chi_A(\au,\q,M;n)$ for $n \gg 0$. In case of $M=A$, Bo\u{d}a-Schenzel \cite{BSCH} proved that $\chi_A(\au,\q,M;n) \geq 0$ for all $n \gg 0$. This can be generalized for an $A-$module $M$ with slight modification. In the following, in a particular case, we bound this Euler characteristic from the upper side, and also discuss the equality in terms of Cohen-Macauleyness of $M$.

\begin{corollary}\label{cor-ub2}
	With the previous notations, if $\au'^\star=a_1^\star,\ldots,a_{d-1}^\star$ is a $G_M(\q)$-regular sequence, then
	\begin{itemize}
		\item[(1)] $\chi_A(\au,\q,M;n) \leq \ell_A(L_1(\au,\q,M;n))$ for all $n \gg 0$.
		\item[(2)] Equality occurs if and only if $M$ is Cohen-Macauley.
	\end{itemize}
\end{corollary}

\begin{proof}
	Since $\ell_A(M/\au\,M) \geq e_0(\underline{a}; M)$,
	therefore claim in (1) follows from previous theorem \ref{th-mul}. Note that
	\[
	\chi_A(\au,\q,M;n) = \ell_A(L_1(\au,\q,M; n)) \text{ for all } n \gg 0
	\Leftrightarrow \ell_A(M / \au M) = e_0(\au; M), 
	\]
	see theorem \ref{th-mul}.
	The latter is equivalent to the fact that
	$M$ is Cohen-Macaulay, see \cite{BH}. This finishes (2).
\end{proof}

In the following, we discuss a few more properties of Euler characteristic $\chi_A(\au,\q,M;n)$. We need the following lemma.

\begin{lemma}\label{euler1}
	With the previous notations, assume that $a \in \q^c \setminus \q^{c+1}$ such that $\dim M/aM = d-1$. Then the following holds.
	\begin{itemize}
		\item[(1)] If $\dim 0:_M a \leq d-2$, then $c\cdot e_0(\q;M) \leq e_0(\q; M/aM)$. Moreover, equality occurs if and only if $\deg \ell_A(\q^nM :a / (\q^{n-c}M + 0:_Ma)) \leq d-2$ for all $n \gg 0$.
		\item[(2)] If $\dim 0:_M a = d-1$, then $c\cdot e_0(\q;M) + e_0(\q; 0:_Ma) \leq e_0(\q; M/aM)$. Moreover, equality occurs if and only if $\deg \ell_A(\q^nM :a / (\q^{n-c}M + 0:_Ma)) \leq d-2$ for all $n \gg 0$.
		\end{itemize}
\end{lemma}

\begin{proof}
	Note that we have the following complex
	\[
	\mathcal{L}_{\bullet}(a,\q,M;n) : 0 \to M/\q^{n-c}M \stackrel{a} \to M/ \q^nM \to 0,
	\]
	and hence
	\[
	(\star) \hspace{.5cm} \ell_A(M/ \q^nM) - \ell_A(M/\q^{n-c}M) = \ell_A(L_0(a,\q,M;n)) - \ell_A(L_1(a,\q,M;n)),
	\]
	where $L_0(a,\q,M;n) \cong M / (a,\q^n)M$ and $L_1(a,\q,M;n) \cong \q^nM : a / \q^{n-c}M$. We break\\ $\ell_A(\q^nM : a / \q^{n-c}M)$ by using the following short exact sequence:
	\[
	0 \to (\q^{n-c}M + 0:_Ma) / \q^{n-c}M \to \q^nM : a / \q^{n-c}M \to \q^nM :a / (\q^{n-c}M + 0:_Ma) \to 0.
	\]
	Also, by Artin-Rees we have
	\[
	(\q^{n-c}M + 0:_Ma) / \q^{n-c}M = 0:_M a / \q^{n-c}M \cap 0:_Ma = 0:_Ma / \q^{n-c-l}(\q^lM \cap 0:_M a)
	\]
	for some $l \in \mathbb{N}$ and for all $n \geq l$.
	That is,
	\begin{gather*}
	\ell_A(\q^nM : a / \q^{n-c}M) = \ell_A(\q^lM \cap 0:_M a / \q^{n-c-l}(\q^lM \cap 0:_M a))\\ + \ell_A(\q^nM :a / (\q^{n-c}M + 0:_Ma)) + \ell_A(0:_M a / 0:_M a \cap \q^lM).
	\end{gather*}
	By using last equation into $(\star)$, we get
	\begin{gather*}
	\ell_A(M/ \q^nM) - \ell_A(M/\q^{n-c}M) = \ell_A(M / (a,\q^n)M) -  \ell_A(\q^lM \cap 0:_M a / \q^{n-c-l}(\q^lM \cap 0:_M a))\\ - \ell_A(\q^nM :a / (\q^{n-c}M + 0:_Ma)) - \ell_A(0:_M a / 0:_M a \cap \q^lM),
	\end{gather*}
	where all lengths involved are polynomials for $n \gg 0$ with $\deg(\ell_A(M/ \q^nM) - \ell_A(M/\q^{n-c}M)) = \deg \ell_A(M / (a,\q^n)M) = d-1$, $\deg \ell_A(\q^lM \cap 0:_M a / \q^{n-c-l}(\q^lM \cap 0:_M a)) = \dim 0:_M a \leq d-1$, $\deg \ell_A(\q^nM :a / (\q^{n-c}M + 0:_Ma)) \leq d-1$ and $\ell_A(0:_M a / 0:_M a \cap \q^lM)$ is a constant. Also, leading terms of $\ell_A(M/ \q^nM) - \ell_A(M/\q^{n-c}M)$ and $\ell_A(M / (a,\q^n)M)$ are $c \cdot e_0(\q;M)$ and $e_0(\q;M/aM)$ respectively for all $n \gg 0$. Now, in case of (1), we get
	\[
	c\cdot e_0(\q;M) \leq e_0(\q;M/aM),
	\]
	and in case of (2), we get
	\[
	c\cdot e_0(\q;M) + e_0(\q;0:_M a) \leq e_0(\q;M/aM).
	\]
	Note that leading term of $\ell_A(\q^lM \cap 0:_M a / \q^{n-c-l}(\q^lM \cap 0:_M a))$ is $e_0(\q; \q^lM \cap 0:_Ma),$
	which is equal to $e_0(\q;0:_Ma)$. Indeed, we have the following short exact sequence
	\[
	0 \to \q^lM \cap 0:_M a \to 0 :_M a \to 0 :_M a / (\q^lM \cap 0:_M a) \to 0,
	\]
	and $\dim 0 :_M a / (\q^lM \cap 0:_M a) = 0$ whereas $\dim 0 :_M a = \dim (\q^lM \cap 0 :_M a)$. Therefore $e_0(\q;0:_Ma) = e_0(\q; \q^lM \cap 0:_Ma)$, cf. \cite[Theorem 13.3]{hM}. 
	Finally, equality in both cases occur if and only if $\deg \ell_A(\q^nM :a / (\q^{n-c}M + 0:_Ma)) \leq d-2$ for all $n \gg 0$.
\end{proof}
	
There is the following consequence of previous lemma.

\begin{proposition}\label{euler2}
	With the previous notations, let $\au=a_1,\ldots,a_d$ be a system of parameters of $M$, then the following holds.
	\begin{itemize}
		\item [(1)] If $\dim 0:_M a \leq d-2$, then $\chi_A(\au,\q,M) \geq \chi_A(\au',\q,M/a_1M)$, where $\au'=a_2,\ldots,a_d$. Moreover, equality occurs if and only if 
		\[
		\deg \ell_A(\q^nM :a_1 / (\q^{n-c_1}M + 0:_Ma_1)) \leq d-2
		\]
		for all $n \gg 0$.
		\item [(2)] If $\dim 0:_M a = d-1$, then $\chi_A(\au,\q,M) + \chi_A(\au',\q,0:_Ma_1) \geq \chi_A(\au',\q,M/a_1M)$, where $\au'=a_2,\ldots,a_d$. Moreover, equality occurs if and only if 
		\[
		\deg \ell_A(\q^nM :a_1 / (\q^{n-c_1}M + 0:_Ma_1)) \leq d-2
		\]
		for all $n \gg 0$.
	\end{itemize}
\end{proposition}

\begin{proof}
	Note that $\au'=a_2,\ldots,a_d$ is a system of parameters for both $A$-modules $M /a_1M$ and $0:_M a_1$. Also, it is a well known fact that
	\[
	e_0(\au;M) + e_0(\au'; 0:_Ma_1)=e_0(\au';M/a_1M),
	\]
	see for example \cite{BH}. For (1), we use Lemma \ref{euler1}(1) and get $c \cdot e_0(\q;M) \leq c_2\cdot \ldots \cdot c_d \cdot e_0(\q;M/a_1M)$, where $c=c_1\cdot \ldots \cdot c_d$. Hence by definition
	\[
	\chi_A(\au,\q,M) \geq \chi_A(\au',\q,M/a_1M),
	\]
	where the equality occurs if and only if $\deg \ell_A(\q^nM :a / (\q^{n-c_1}M + 0:_Ma_1)) \leq d-2$ for all $n \gg 0$.
	
	For (2),  we use Lemma \ref{euler1}(2) and get 
	\[
	c \cdot e_0(\q;M) + c_2\cdot \ldots \cdot c_d\cdot e_0(\q;0:_Ma_1) \leq c_2\cdot \ldots \cdot c_d \cdot e_0(\q;M/a_1M).
	\]
	Hence by definition
	\[
	\chi_A(\au,\q,M) + \chi_A(\au',\q,0:_Ma_1) \geq \chi_A(\au',\q,M/a_1M),
	\]
	where the equality occurs if and only if $\deg \ell_A(\q^nM :a / (\q^{n-c_1}M + 0:_Ma_1)) \leq d-2$ for all $n \gg 0$.
\end{proof}

We finish with the following remark.

\begin{remark}\label{euler3}
	Lemma \ref{euler1}, with slight modification, originally proved by Flenner-Vogel \cite{FV}. More precisely, they proved the equality in lemma if and only if $a^\star$ is a parameter for $G_M(\q)$. The author of present note tried to prove directly that $\deg \ell_A(\q^nM :a / (\q^{n-c_1}M + 0:_Ma)) \leq d-2$ for all $n \gg 0$ if and only if $a^\star$ is a parameter for $G_M(\q)$. The "if" part is easy. Indeed, $\deg \ell_A(\q^nM :a / (\q^{n-c_1}M + 0:_Ma)) \leq d-2$ for all $n \gg 0$ implies that 
	\[
	\deg \ell_A(\q^nM : a / \q^{n-c_1}M + (\q^{n+1}M : a)) \leq d-2 \text{ for all } n \gg 0,
	\]	
	where
	\[
	\q^nM : a / \q^{n-c_1}M + (\q^{n+1}M : a) \cong [\ker (G_M(\q) / a^\star G_M(\q) \to G_{M/aM}(\q))]_n.
	\]
	But, this is equivalent to $\dim (\ker (G_M(\q) / a^\star G_M(\q) \to G_{M/aM}(\q))) \leq d-1$ which is equivalent to the fact that $a^\star$ is a parameter for $G_M(\q)$.
	\end{remark}

{\bf Acknowledgement:} The author is grateful to the reviewer for comments and suggestions.

\end{document}